\documentclass[11pt]{article}

\usepackage{amsmath}
\usepackage{amssymb}  
\usepackage{epsfig}
\usepackage{amsthm}   
\usepackage[mathcal]{eucal}

\usepackage{url}

\newcommand{\BR}{\bar{\mathbb{R}}}

\newcommand{\inner}[2]{\langle{#1},{#2}\rangle}
\newcommand{\linner}[2]{ {\langle{#1},{#2}\rangle}_\lambda}

\newcommand{\Inner}[2]{\left\langle{#1},{#2}\right\rangle}
\newcommand{\norm}[1]{\|#1\|}
\newcommand{\lnorm}[1]{ {\|#1\|}_\lambda}
\newcommand{\Lnorm}[1]{ {\left\|#1\right\|}_\lambda}
\newcommand{\Norm}[1]{\left\|#1\right\|}

\newcommand{\tos}{\rightrightarrows} 

\newtheorem{theorem}{Theorem}
\newtheorem{lemma}[theorem]{Lemma}
\newtheorem{corollary}[theorem]{Corollary}

\title{Weak convergence on Douglas-Rachford method}

\author{B. F. Svaiter\thanks{ IMPA, 
Estrada Dona Castorina 110, 22460-320 Rio de
    Janeiro, Brazil ({\tt benar@impa.br}) 
    tel: 55 (21) 25295112, fax: 55 (21)25124115.  }\hspace{.5em}
    \thanks{Partially supported by CNPq
    grants 300755/2005-8, 475647/2006-8 and by PRONEX-Optimization}
}

\date{}

\begin{document}

\maketitle

\begin{abstract}
  We prove that the sequences generate by the Douglas-Rachford method
  converge weakly to a solution of the inclusion problem
  \\
  \\
\end{abstract}

\pagestyle{plain}



From now on $H$ is a real Hilbert space with inner product
$\inner{\cdot}{\cdot}$ and associated norm $\norm{\cdot}$.
A \emph{point-to-set operator} $T:H\tos H$ is a
relation
\(
  T\subset H\times H
\)
and for $x\in H$,
\[
T(x)=\{v\in H\;|\; (x,v)\in T\}.
\]
An operator $T:H\tos H$ is \emph{monotone} if
\[ \inner{x_1-x_2}{v_1-v_2}\geq 0, \forall (x_1,v_1),(x_2,v_2)\in T
\]
and it is \emph{maximal monotone} if it is monotone and maximal in the family
of monotone operators with respect to the partial order of inclusion.

Let $A,B$ be maximal monotone operators in $H$. Consider the problem of
finding $x\in H$ such that
\begin{equation}
  \label{eq:p}
  0\in A(x)+B(x).
\end{equation}
Douglas-Rachford method (with sumable residual error) works as follows. Let
$\{\alpha_k\}$, $\{\beta_k\}$ be sequences of positive error tolerance such that
\[
\sum_{k=1}^\infty \alpha_k<\infty, 
  \quad \sum_{k=1}^\infty \beta_k<\infty.
\]
Take $\lambda>0$,
$(x_0,b_0)\in B $ and for $k=1,2,\dots$
\begin{description}
\item [(a)] Find  $(y_{k},a_{k})$ such that 
  \begin{equation}
    \label{eq:s1}
    (y_{k},a_{k})\in A,\quad
   \Norm{ y_{k}+\lambda a_{k}-\left(x_{k-1}-\lambda b_{k-1}\right)}
    \leq \alpha_k
  \end{equation}
\item  [(b)] Find $(x_{k},b_{k})$ such that
  \begin{equation}
    \label{eq:s2}
    (x_{k},b_{k})\in B,\quad
    \Norm{x_{k}+\lambda b_{k}-\left( y_{k}+\lambda b_{k-1}\right)}\leq \beta_k
  \end{equation}
\end{description}
Existence of $(y_k,a_k)$, $(x_k,b_k)$ as above follows from Minty's
Theorem~\cite{Minty}.  From now on $\{(x_k,b_k)\}_{k=0} ^\infty$,
$\{(y_k,a_k)\}_{k=1} ^\infty$ are sequences generated by the above
algorithm.
This algorithm was proposed by Douglas and Rachford~\cite{DR} for the case
where $A$ and $B$ are linear, and was extended to arbitrary maximal monotone
operators by Lions and Mercier~\cite{LM}.
These later authors proved that $x_k+\lambda b_k$ converges weakly
to a point $\bar z$ such that, for some $\bar x$, $\bar b$,
\[ \bar z=\bar x+\lambda \bar b,\quad \bar b\in B(\bar x),\quad
-\bar b\in A(\bar x).
\]
Our aim is to prove the next theorem
\begin{theorem}
  \label{th:main}
  If $A,B$ are maximal monotone operators and the solution set of
  \[
  0\in A( x)+ B( x)
  \]
  is non-empty, then the sequences $\{(x_k,b_k)\}$ and $\{(y_k,a_k)\}$
  generated by Douglas-Rachford method converges
  weakly to some $(\bar x,\bar b)$ and $(\bar x, -\bar b)$ respectively,
  such that
  \[ \bar b\in B(\bar x),\qquad -\bar b\in A(\bar x)
  \]
  and hence, $0\in A(\bar x)+B(\bar x)$.
\end{theorem}

For each $k$, there exists  unique pairs $(\hat y_k,\hat a_k)$,
$(\hat x_k,\hat b_k)$ such that
\begin{align}
\label{eq:e1}
(\hat y_k,\hat a_k)&\in A,\quad
\hat y_{k}+\lambda \hat a_{k}=x_{k-1}-\lambda b_{k-1}\\
\label{eq:e2}
(\hat x_k,\hat b_k)&\in B,\quad
\hat x_{k}+\lambda \hat b_{k}=\hat y_{k}+\lambda b_{k-1}
\end{align}
Existence of $(\hat y_k,\hat a_k)$, $(\hat x_k,\hat b_k)$ as above follows from Minty's
Theorem~\cite{Minty}.
Direct use of \eqref{eq:e1}  and 
\eqref{eq:e2} shows that for $k=1,2,\dots$
\begin{equation}
  \label{eq:update}
  x_{k-1}-\hat x_k=\lambda(\hat a_k+\hat b_k),\quad \lambda(b_{k-1}-\hat b_k)=\hat x_k-\hat y_k.
\end{equation}
and
\begin{equation}
  \label{eq:all}
\hat x_k-\hat y_k+\lambda(\hat a_k+\hat b_k)=\lambda(\hat a_k+b_{k-1}) =x_{k-1}-\hat y_k.
\end{equation}
Using \eqref{eq:s1},  \eqref{eq:e1} and the monotonicity of $A$
we conclude that
\begin{align}
  \nonumber
  \norm{y_k-\hat y_k}^2+\lambda^2\norm{a_k-\hat a_k}
    \leq&\norm{y_k-\hat y_k+\lambda(a_k-\hat a_k)}^2\\
  \label{eq:ye}
    =   &\norm{y_k+\lambda a_k-(x_{k-1}-\lambda b_{k-1})}^2
    \leq \alpha_k^2
\end{align}
Using \eqref{eq:e2},  \eqref{eq:s2}, the monotonicity of $B$,
triangle inequality 
and  the above inequality we have
\begin{align}
  \nonumber
 \norm{x_k-\hat x_k}^2+\lambda^2\norm{b_k-\hat b_k}
   \leq&\Norm{x_{k}-\hat x_k+\lambda( b_{k}-\hat b_k)}^2\\
  \nonumber
   \leq&\Norm{x_{k}+\lambda b_{k}-\left( \hat y_{k}+\lambda b_{k-1}\right)}^2\\
  \nonumber
   \leq&\left(
    \Norm{x_{k}+\lambda b_{k}-\left( y_{k}+\lambda b_{k-1}\right)}
      +\Norm{\hat y_k-y_k}\right)^2\\
  \label{eq:ps0}
   \leq & (\alpha_k+\beta_k)^2
\end{align}

We will use in $H\times H$ the inner product $\linner \cdot \cdot$ and
associated norm $\lnorm \cdot$,
\[
\linner{(x,v)}{(x',v')}=\inner{x}{x'}+\lambda^2\inner{v}{v'},
\quad \lnorm{p}=\sqrt\linner p p.
\]
Note that $H\times H$ endowed with the inner product $\linner \cdot\cdot$
is a Hilbert space isomorphic to $H\times H$ endowed with the canonical 
inner product  $\inner{(x,v)}{(x',v')}=\inner{x}{x'}+\inner{v}{v'}$.
Hence, the  strong/weak topologies of both spaces are the same.
To simplify the exposition, define
\begin{equation}
  \label{eq:def.ps}
  p_k=(x_k, b_k),\quad k=0,1,\dots,\qquad
  \hat p_k=(\hat x_k, \hat b_k),\quad k=1,2,\dots
\end{equation}
We have just proved in \eqref{eq:ps0} that
\begin{equation}
  \label{eq:pp}
  \Lnorm{p_k-\hat p_k}\leq \alpha_k+\beta_k,\qquad k=1,2,\dots
\end{equation}
The \emph{extended solution set}~\cite{ES1} of problem \eqref{eq:p} 
is
\[
S(A,B)= B\cap-A=\{ (z,w)\;|\; (z,w)\in B, 
   (z,-w)\in A\}
\]
\begin{lemma}
  \label{lm:bas}
  If $ p\in S(A,B)$ then, 
  \begin{align*}
  \lnorm{p_{k-1}- p}^2\geq
  \lnorm{\hat p_k- p}^2+\norm{x_{k-1}-\hat y_k}^2
   =\lnorm{\hat p_k- p}^2+\lambda^2\norm{\hat a_k+b_{k-1}}^2
\end{align*}
for $k=1,2,\dots$.
\end{lemma}
\begin{proof}
   First note that $p=(x,b)$ with $b\in B(x)$ and $-b\in A(x)$. Using
   the first equality in \eqref{eq:update}, the monotonicity of $B$ and
 the monotonicity of $A$ we conclude that
  \begin{align*}
   \inner{x_{k-1}-\hat x_k}{\hat x_k- x}=&
     \lambda\inner{\hat a_k+\hat b_k}{\hat x_k- x}\\
    =&\lambda\bigg[
       \inner{\hat a_k+ b}{\hat x_k- x}+\inner{\hat b_k- b}{\hat x_k- x}
    \bigg]\\
   \geq& \lambda \inner{\hat a_k+ b}{\hat x_k- x}\\
   =&\lambda\bigg[ \inner{\hat a_k+ b}{\hat x_k-\hat y_k}+
     \inner{\hat a_k+ b}{\hat y_k-  x}
   \bigg]\geq \lambda\inner{\hat a_k+ b}{\hat x_k-\hat y_k}
  \end{align*}
  Using the above inequality, the second equality in \eqref{eq:update}
  and \eqref{eq:def.ps} we have
\begin{align*}
  \linner{p_{k-1}-\hat p_k}{\hat p_k- p}=&\inner{x_{k-1}-\hat x_k}{\hat x_k- x}
    +\lambda^2\inner{b_{k-1}-\hat b_k}{\hat b_k- b}\\
   \geq & \lambda\inner{\hat a_k+ b}{\hat x_k-\hat y_k}
   +\lambda\inner{\hat x_k-\hat y_k}{\hat b_k- b}\\
  =&\lambda\inner{\hat x_k-\hat y_k}{\hat a_k+\hat b_k}
\end{align*}
Using \eqref{eq:def.ps} and \eqref{eq:update} we have
  \[ \lnorm{p_{k-1}-\hat p_k}^2=\lambda^2\norm{\hat a_k+\hat b_k}^2
      +\norm{\hat x_{k}-\hat y_k}^2
  \]
 Therefore, 
  \begin{align*}
    \lnorm{p_{k-1}- p}^2=&\lnorm{p_{k-1}-\hat p_k}^2
     +2\linner{p_{k-1}-\hat p_k}{\hat p_k- p}+ \lnorm{\hat p_k- p}^2\\
   \geq&
     \lambda^2\norm{\hat a_k+\hat b_k}^2
      +\norm{\hat x_{k}-\hat y_k}^2
     +2\lambda\inner{\hat x_{k}-\hat y_k}{\hat a_k+\hat b_k}\\
     &\mbox{}+
   \lnorm{\hat p_k-p}^2
      \\
     =&\lnorm{\hat p_k-p}^2
  +\norm{\hat x_k-\hat y_k+\lambda(b_k+a_k)}^2
\end{align*}
To end the proof, use the above inequality and \eqref{eq:all}.
\end{proof}
\begin{corollary}
  \label{cr:qf}
  The sequence $\{p_k\}$ is Quasi-Fejer convergent to $S(A,B)$.
  Therefore it has at most one weak cluster point in this set, and it is bounded if $S(A,B)\neq\emptyset$.
\end{corollary}
\begin{proof}
   Take $p\in S(A,B)$. 
  Using \eqref{eq:pp} and Lemma~\ref{lm:bas} we have
  \[ \norm{p_k- p}\leq \norm{p_k-\hat p_k}+\norm{\hat p_k- p}\leq
  \alpha_k+\beta_k+\norm{p_{k-1}- p}
  \]
  which proves that  $\{p_k\}$ is Quasi-Fejer convergent to
  $S_\lambda(A,B)$. The last part of the corollary follows from this
  result and also from Opial's Lemma~\cite{OP}.
\end{proof}

\begin{proof}[Proof of Theorem \ref{th:main}]
  We are assuming that problem~\eqref{eq:p} has a solution. Therefore, 
 $S(A,B)\neq\emptyset$. 
 Using Corollary~\ref{cr:qf} we conclude that $\{p_k\}$ is bounded. 
 Take
 \[
 p\in S(A,B)
 \]
 and let
 \[
 M=1+\sup \lnorm{p_k- p}
 \]
 Using Lemma~\ref{lm:bas} we have, for $k=1,2,\ldots$
 \[ \lnorm{p^*_k-p}^2\leq  \lnorm{p_{k-1}-p}^2-\norm{x_{k-1}-y^*_k}^2.
 \]
 Therefore, using also the concavity of $t\mapsto \sqrt{t}$ we conclude that
 \[ \lnorm{p^*_k-p} \leq \lnorm{p_{k-1}- p}
    -\frac{1}{2M}\norm{x_{k-1}-\hat y_k}^2
\]
Hence, combining the above inequality with \eqref{eq:pp} and triangle
inequality we obtain 
 \begin{align*}
   \lnorm{p_k- p}
   \leq& \lnorm{p_k-\hat p_k}+\lnorm{\hat p_k- p}\\
   \leq &(\alpha_k+\beta_k)+\lnorm{p_{k-1}- p}
    -\frac{1}{2M}\norm{x_{k-1}-\hat y_k}^2
 \end{align*} 
 Adding the above inequality for $k=1,2,\dots,n$ we conclude that
 \[ \frac{1}{2M}\sum_{k=1}^n\norm{x_{k-1}-\hat y_k}^2
    \leq \lnorm{p_0- p}+ \sum_{k=1}^n (\alpha_k+\beta_k)
\]
Therefore
\[ \sum_{k=1}^\infty \norm{x_{k-1}-\hat y_k}^2 <\infty
\]
and, using also \eqref{eq:all}, we conclude that 
\begin{equation*}
  \lim_{k\to\infty} x_{k-1}-\hat y_k =\lim_{k\to\infty} \hat a_k+b_{k-1}=0
\end{equation*}
Using \eqref{eq:ye} we have
\[ 
 \lim_{k\to\infty} y_k-\hat y_k =\lim_{k\to\infty} a_k- \hat a_k=0.
\]
Therefore
\begin{equation}
  \label{eq:conv.0}
 \lim_{k\to\infty} x_{k-1}-y_k =\lim_{k\to\infty} a_k+b_{k-1}=0
\end{equation}
and sequence $\{(y_k,b_k)\}$ is also bounded.

Since the sequence $\{p_k\}$ is bounded and $H\times H$
(endowed with the norm $\lnorm\cdot$)
is reflexive, this sequence
has weak cluster points.
Let $(\bar x,\bar b)$ be a weak cluster point of
the \emph{bounded sequence} $\{p_k=(x_k, b_k)\}$.
Using again the fact that  
$H\times H$  is reflexive, we conclude that 
there exists a subsequence $\{ (x_{k_j},b_{k_j})\}$ converging
weakly to $(\bar x,\bar b)$
and hence
\[
x_{k_j} \stackrel w \to \bar x,\;\;
b_{k_j} \stackrel w \to \bar b,\quad \mbox{ as } j\to\infty,
\]
which, together with \eqref{eq:conv.0} implies also that
\[
y_{k_j-1} \stackrel w \to \bar x,\;\;
a_{k_j-1} \stackrel w \to -\bar b,\quad \mbox{ as } j\to\infty.
\]
Using the two above equations, \eqref{eq:conv.0} and 
Lemma~\ref{lm:aux} (see Appendix~\ref{sec:app}) applied to the subsequences
 $\{(x_{k_j},b_{k_j}\}$, $\{(y_{k_j-1},a_{k_j-1}\}$,
we conclude that $(\bar x,\bar b)\in B$,
$(\bar x,-\bar b)\in A$, that is, $(\bar x,\bar b)\in S(A,B)$.

We have proved that the sequence $\{p_k\}$ has weak cluster points and
that all these weak cluster points are in $S(A,B)$. Using these results
and Corollary~\ref{cr:qf} we conclude that $\{p_k\}$ has only one weak
cluster point $(\bar x,\bar b)$, and  this (weak cluster) point belongs
to $S(A,B)$.  As $\{p_k\}$ is bounded and $H\times H$ is reflexive, the
sequence $\{p_k\}$ converges weakly to such point $(\bar x,\bar b)$,
which is equivalent 
\[
x_{k} \stackrel w \to \bar x,\;\;
b_{k} \stackrel w \to \bar b,\quad \mbox{ as } k\to\infty,
\]
To end the proof, use the above equation and~\eqref{eq:conv.0} to
conclude that
\[
y_{k} \stackrel w \to \bar x,\;\;
a_{k} \stackrel w \to -\bar b,\quad \mbox{ as } k\to\infty.
\]
\end{proof}

The convergence analysis presented is are based on the framework and
techniques introduced in \cite{ES1,ES2}, and becomes more intuitive
using this framework and results.
To make this note shorter, we do not presented this framework here.
For historical reasons, here we
used the classical sumable error tolerance. 
\appendix

\section{An auxiliary result}
\label{sec:app}

Let $X$ be a real Banach space with topological dual $X^*$. For $x\in X$,
$x^*\in X^*$ we use the notation $\inner{x}{x^*}=x^*(x)$.
An operator $T:X\tos X^*$ is monotone if 
$\inner{x-y}{x^*-y^*}\geq 0$ for all $(x,x^*),(y,y^*)\in T$, and is maximal monotone
if it is monotone and  maximal in the family
of monotone operators with respect to the partial order of inclusion.
\begin{lemma}
  \label{lm:aux.0}
  Let $X$ be a real Banach space with topological dual $X^*$. If
  $T:X\tos X$ is maximal monotone, $\{(x_i,x^*_i)\}_{i\in I}$ is a net
  in $T$ which converges in the weak$\times$weak$*$ topology to $(\bar
  x,\bar x^*)$, then
  \[ 
  \lim\inf_{i\to\infty}\inner{x_i}{x^*_i}\geq
     \inner{\bar x}{\bar x^*}.
  \]
  Moreover, if the above inequality holds as an equality, then
  $(x,x^*)\in T$.
\end{lemma}
\begin{proof}
Let $\varphi:X\times X^*\to\BR$,
\[ \varphi(x,x^*)=\sup_{(y,y^*)\in T}\inner{x}{y^*}+\inner{y}{x^*}
    -\inner{y}{y^*}
\]
The function $\varphi$ is Fitzpatrick minimal function~\cite{Fitz} of
$T$, it is lower semicontinuous in the weak$\times$weak$*$ topology,
$\varphi(x,x^*)\geq\inner{x}{x^*}$
for all $(x,x^*)$ and this inequality holds as an equality if and only
if $(x,x^*)\in T$. Therefore $\inner{x_i}{x^*_i}=\varphi(x_i,x^*_i)$
for all $i\in I$ and
\[
\lim\inf_{i\to\infty}\inner{x_i}{x^*_i}
= \lim\inf_{i\to\infty}\varphi(x_i,x^*_i)\geq\varphi(\bar x,\bar x^*)
\geq \inner{\bar x}{\bar x^*}
\]
To end the proof, use the fact that $\varphi$ is bounded below by the
duality product and coincide with the duality product if and only if
$(x,x^*)\in T$.
\end{proof}

\begin{lemma}
  \label{lm:aux}
Let $X$ be real Banach space. If
$T_1,\dots,T_m:X\tos X^*$ are maximal monotone operators
and $\{(x_{k,i},x^*_{k,i})\}_{i\in I}$ are bounded nets such that
$(x_{k,i},x^*_{k,i})\in T_k$ for all $k=1,\dots,m$ $i\in I$,
and
\begin{align*}
x_{k,i}-x_{j,i}&\to 0\;\;\; j,k=1,\dots,m &
\sum_{k=1}^mx^*_{k,i}&\to \bar x^* \\
x_{k,i}&\stackrel w \to \bar x,&
    x^*_{k,i}&\stackrel {w*} \to {\bar x_k}^*\qquad k=1,\dots,m
\end{align*}
 as $i\to\infty$, then $ (\bar x,{\bar x_k}^*)\in T_k$ for $k=1,\dots,m$.
\end{lemma}
\begin{proof}
  In view of the above assumptions,
  \[ \sum_{k=1}^m  {\bar x_k}^*=\bar x^*
  \]
  Define
  \[
  \alpha_{k,i}=\inner{x_{k,i}}{x^*_{k,i}}-\inner{\bar x}{\bar x^*_k},
  \quad k=1,\dots,m\;i\in I.
  \]
  Direct algebraic manipulations yield
  \begin{align*}
    \sum_{k=1}^m\alpha_{k,i}=&\left(
    \sum_{k=0}^m\inner{x_{k,i}}{x^*_{k,i}}\right)
    -\inner{\bar x}{\bar x^*}\\
    =&\left(
    \sum_{k=0}^m\inner{x_{k,i}-x_{1,i}}{x^*_{k,i}}\right)
    +\Inner{x_{1,i}}{   \left( \sum_{k=0}^mx^*_{k,i}\right)-\bar x^*}
    +\inner{x_{1,i}-\bar x}{\bar x^*},
  \end{align*}
  which ready implies, in view of the assumptions of the lemma, that
  \[
  \lim_{i\to\infty} \sum_{k=1}^m\alpha_{k,i}=0
  \]
  Using the first part of Lemma~\ref{lm:aux.0} we have
  \[ \lim\sup_{i\to\infty}\alpha_{k,i}\geq 0, \qquad k=1,\dots,m.
  \]
  Combining the two above equations we conclude that
  \(\lim_{i\to\infty} \alpha_{k,i}=0\) for $k=1,\dots,m$,
  that is
  \[ \lim_{i\to\infty}\inner{x_{k,i}}{x^*_{k,i}}
     =\inner{\bar x}{\bar x^*_k}
  \]
  To end the proof, use the above equation and the second part of
  Lemma~\ref{lm:aux.0}.
\end{proof}

\bibliographystyle{plain}

\end{document}